\documentclass[12pt,letterpaper]{amsart}
\usepackage{geometry,xcolor}
\usepackage{amssymb,setspace}
\usepackage{mathpazo} 
\usepackage[pdftex,breaklinks,colorlinks,
citecolor=blue,
urlcolor=blue,
linkcolor=blue]{hyperref}
\newtheorem{theorem}{Theorem}
\newtheorem{proposition}{Proposition}
\newtheorem{lemma}{Lemma}
\newtheorem{corollary}{Corollary}
\newtheorem{definition}{Definition}
\theoremstyle{remark}
\newtheorem{remark}{Remark}

\newcommand{\C}{\mathbb{C}}
\newcommand{\disk}{\mathbb{D}}
\newcommand{\D}{\Omega}
\newcommand{\ep}{\varepsilon}

\newcommand{\ab}{\overline{a}}

\newcommand{\wb}{\overline{w}}
\onehalfspace

\title{On convergence of the Berezin transforms}

\author{N\.{i}hat G\"{o}khan G\"{o}\u{g}\"{u}\c{s}}
\address{Sabanc\i{} University,
Tuzla, 34956, Istanbul, Turkey}
\email{nggogus@sabanciuniv.edu}
\thanks{This work was completed with the support of a 
	TUBITAK project with project number 118F405.}

\author{S\"{o}nmez \c{S}ahuto\u{g}lu}
\address{University of Toledo, Department of
Mathematics \& Statistics, Toledo, OH 43606, USA}
\email{sonmez.sahutoglu@utoledo.edu}


\date{\today}
\begin{document}

\begin{abstract}
We prove approximation results about sequences of Berezin transforms
of finite sums of finite product of Toeplitz operators (and bounded linear
maps, in general)  in the spirit of  Ramadanov and Skwarczy\'nski Theorems
that are about convergence of Bergman kernels.
\end{abstract}

\maketitle

Let $\D$ be a domain in $\C^n$ and $A^2(\D)$ denote the Bergman space,
the set of square integrable holomorphic functions, of $\D$. Since the
Bergman space $A^2(\D)$ is a closed subspace of $L^2(\D)$, there exists
a bounded orthogonal projection $P_{\D}$ from $L^2(\D)$ onto $A^2(\D)$.
This is called the Bergman projection for $\D$. We denote the Bergman
kernel of $\D$ by $K^{\D}$. The Berezin transform $B_{\D}T$ of
a bounded linear operator $T$ on $A^2(\D)$ is defined as
\[B_{\D}T(z)=\langle Tk^{\D}_z,k^{\D}_z\rangle,\]
where $k^{\D}_z(\xi)=K^{\D}(\xi,z)/\sqrt{K^{\D}(z,z)}$ is the
normalized Bergman kernel of $\D$ and $\langle .,.\rangle$ denotes the
inner product on $A^2(\D)$.

Berezin transform is an important notion in operator theory. For instance,
it is used to characterize compactness  of operators in the Toeplitz algebra
on the unit disc and the unit ball (see \cite{AxlerZheng98,Suarez07}) and
in a subalgebra on more general domains in $\C^n$
(see \cite{CuckovicSahutoglu13,CuckovicSahutogluZeytuncu18}).
Berezin transform is also an important tool in the characterization
of compactness of the Hankel operators in \cite{BBCZ90}.

There are different notions for convergence of operators on $A^2(\D)$.
For instance,  one can ask if a sequence of bounded operators defined
on the same Bergman space  converges to a bounded operator in the
operator norm or in the weak sense. Now assume that, for each $j$, $T_j$ is a
bounded operator on $A^2(\D_j)$  and $\D_j\subset \D$ (or $\D\subset \D_j$).
Since the operators $T_j$s are defined on different spaces
it does not make sense to talk about  convergence of $\{T_j\}$
in norm or weakly. However, we can compare Berezin transforms.
That is, we can ask if $\{B_{\D_j}T_j\}$ converges to $B_{\D}T$ pointwise,
locally uniformly, etc. This notion generalizes the weak
convergence of operators because  $B_{\D}T_j\to B_{\D}T$ pointwise
on $\D$ whenever $T_j$s are defined on $A^2(\D)$ and $T_j\to T$ weakly.

 Let $\{\D_j\}$ be an increasing sequence of domains whose
 union is $\D$. Ramadanov showed that (see
 \cite{Ramadanov67,Ramadanov83}) the Bergman kernels
$\{K^{\D_j}\}$  converge to $K^{\D}$ uniformly on compact subsets of
$\D\times \D$.  In this paper we prove results in the spirit of Ramadanov's
result for Berezin transforms of bounded operators on the Bergman space.

The plan of the paper is as follows: In the next section we will state our
main results. The proofs will be presented in the following section.

\section{Main Results}
To state our results we need to define the restriction operator.
Let $U\subset \D$ be domain in $\C^n$ and $R^{\D}_U:A^2(\D)\to A^2(U)$
denote the restriction operator. That is, $R^{\D}_Uf=f|_U$. Then
the adjoint $R^{\D *}_U:A^2(U)\to A^2(\D)$ of $R^{\D}_U$ is a bounded linear map
and one can show that (see, for example, \cite{ChakrabartiSahutogluPreprint})
\[R^{\D *}_Uf(z)=\int_UK^{\D}(z,w)f(w)dV(w),\]
where $dV$ is the Lebesgue measure in $\C^n$. We note that
if $\overline{U}\subset\D$, then Montel's Theorem implies
that $R^{\D}_U$ is compact. Also $R_U^*TR_U$ is a bounded linear operator
on $A^2(\D)$ whenever $T$ is a bounded linear map on $A^2(U)$.

Throughout this paper $Ef$ denotes the extension of $f$ onto $\C^n$
trivially by zero and $R_U$ will denote $R_U^{\D}$ when the domain
$\D$ is clear from the context. Then the formula for $R^{\D *}_U$
above is $R^{\D *}_U=P_{\D}E$.

For $z,$ $w\in \D$, let $K_z^{\D}(w)=K^{\D}(w,z)$.
Notice that the normalized Bergman kernel $k_z^{\D}$ is well-
defined whenever $K^{\D}(z,z)\neq 0$. In \cite{Englis07}, Engli\v{s}
observes that there are unbounded domains in $\mathbb{C}^{n}$ for
which the zero set $\mathcal{Z}$ of the Bergman kernel on the
diagonal $K^{\D}(z,z)$ is not empty. Namely, we denote
\[ \mathcal{Z}=\left\{ z\in \D :K^{\D}(z,z)=0\right\}.\]
\begin{definition}
A domain $\D$ in $\C^n$ is called a non-trivial Bergman domain if
$A^2(\D)\neq \{0\}$.
\end{definition}
We note that $\D$ is a non-trivial Bergman domain if and only if
$\mathcal{Z}\neq \D$. If $\D$ is bounded, then $\mathcal{Z}$ is
empty because the constant functions
belong to $A^2(\D)$ and $K^{\D}(z,z) \geq 1/\|1\|^2>0$ for all $z\in \D$.
Therefore, bounded domains are non-trivial Bergman domains as well.
The set $\mathcal{Z}$, if not empty and not equal to $\D$,
is a real-analytic variety in $\D$ with zero Lebesgue measure and it is a
relatively closed subset of $\D$. The normalized Bergman kernel
$k^{\D}_z$ is a well defined function in $A^2(\D)$ for $z\in \D\setminus \mathcal{Z}$.
In this paper we will always assume that $\D$ is a non-trivial Bergman domain.

In the example given in \cite{Englis07}, there exists a bounded function $\phi$ on
an unbounded pseudoconvex complete Reinhardt domain
$\D$ such that the Berezin transform $B_{\D}T_{\phi}$ of the (bounded) Toeplitz
operator on $\D$ has a singularity at a point in $\mathcal{Z}$. However, the map
$z\mapsto k^{\D}_z$ is continuous from $\D\backslash \mathcal{Z} $ to
$L^2(\D)$ since
\begin{align}\label{Eqn0}
\left\| k^{\D}_z-k^{\D}_w\right\|_{L^2(\D)}^2
=2-2\text{Re}\left\langle k^{\D}_z,k^{\D}_{w}\right\rangle
=2-2\text{Re}\frac{K^{\D}(w,z)}{\sqrt{K^{\D}(z,z)}\sqrt{K^{\D}(w,w)}}
\end{align}
and both $K^{\D}(w,z)$ and $K^{\D}(w,w)$ converge to
$K^{\D}(z,z)$ as $w$ converges to $z$ in $\D\backslash \mathcal{Z}$.
Hence, the Berezin transform $B_{\D}T$ of a bounded operator $T$ on $A^2(\D)$
is always a well-defined, bounded and continuous function,
on $\D\backslash \mathcal{Z}$. This can be seen from the
inequality $|B_{\D}T(z)|=|\langle Tk^{\D}_z,k^{\D}_z\rangle |\leq \|T\|$ and
\begin{align*}
\left|B_{\D}T(z)-B_{\D}T(w)\right|
\leq &\left|\langle Tk_z,k^{\D}_z-k^{\D}_w\rangle\right|
+\left|\langle T(k^{\D}_z-k^{\D}_w),k^{\D}_w\rangle\right| \\
\leq &\left\|T\right\| \left\|k^{\D}_z-k^{\D}_w\right\|_{L^2(\D)}
+\left\| T(k^{\D}_z-k^{\D}_w)\right\|_{L^2(\D)}
\end{align*}
for every $z$, $w\in\D\backslash \mathcal{Z}$.

Our first two  results below  can be seen as  analogues of
Ramadanov's and Skwarczy\'nski's Theorems.

\begin{theorem}\label{ThmBounded}
Let $\{\D_j\}$ be a sequence of  domains in $\C^n$ such that
$\D_j\subset \D_{j+1}$ for all $j$ and
$\D=\cup_{j=1}^{\infty}\D_j$ be a non-trivial Bergman domain.
Let $T$ be a bounded linear map on $A^2(\D)$. Then
$B_{\D_j}R_{\D_j}TR^*_{\D_j}\to B_{\D}T$ uniformly on compact
subsets of $\D\backslash \mathcal{Z}$ as $j\to \infty$.
Furthermore, if $\D$ is bounded, then
$EB_{\D_j}R_{\D_j}TR^*_{\D_j}\to B_{\D}T$ in $L^p(\D)$ as
$j\to \infty$ for all $0<p<\infty$.
\end{theorem}

\begin{theorem}\label{ThmDecreasing}
Let $\D$ be a non-trivial Bergman domain and $\{\D_j\}$ be a sequence of
domains in $\C^n$ such that $\D\subset\D_{j+1}\subset \D_j$
for all $j$. Assume $K^{\D_j}(z,z)\to K^{\D}(z,z)$ as $j\to\infty$
for every $z\in\D$.  Let $T$ be a bounded linear map on $A^2(\D)$. Then
$B_{\D_j}(R^{\D_j}_{\D})^*TR^{\D_j}_{\D}\to B_{\D}T$ uniformly
on compact  subsets of $\D\backslash \mathcal{Z}$  as $j\to \infty$.
Furthermore, if $\D$ is bounded, then
$B_{\D_j}(R^{\D_j}_{\D})^*TR^{\D_j}_{\D}\to B_{\D}T$ in $L^p(\D)$
as  $j\to \infty$ for all $0<p<\infty$.
\end{theorem}

The next result describes the convergence of the Berezin
transforms when the symbols of Toeplitz operators are restricted
onto the subdomains. To clarify the notation below, $\phi|_U$ denotes
the restriction of $\phi$ onto $U, R_U\phi$.

\begin{theorem}\label{ThmLinfty}
Let $\{\D_j\}$ be a sequence of domains in $\C^n$ such that
$\D_j\subset \D_{j+1}$ for all $j$ and $\D=\cup_{j=1}^{\infty}\D_j$
be a non-trivial Bergman domain.
Assume that  $T=\sum_{m=1}^lT_{\phi_{m,1}}\cdots T_{\phi_{m,k_m}}$
is a finite sum of finite products of Toeplitz operators with
bounded symbols on $\D$ and
$T^{\D_j} =\sum_{m=1}^lT_{\phi_{m,1}|_{\D_j}}\cdots T_{\phi_{m,k_m}|_{\D_j}}$
for each $j$. Then $B_{\D_j}T^{\D_j}\to B_{\D}T$ uniformly on compact
subsets of $\D\setminus \mathcal{Z}$  as $j\to \infty$. Furthermore, 
if $\D$ is bounded, then $EB_{\D_j}T^{\D_j}\to B_{\D}T$ in $L^p(\D)$ as
$j\to \infty$ for all $0<p<\infty$.
\end{theorem}

\begin{remark}
We note that the $T^{\D_j}$ in the theorem above depends on the
symbols and hence representation of $T$. However, representation
of products of Toeplitz operators is not unique.  For instance,
\c{C}elik and Zeytuncu in \cite{CelikZeytuncu16} showed that
there exists a Reinhardt domain $\D$ in $\C^2$ such that there
exists non-trivial nilpotent Toeplitz operators on $A^2(\D)$.
Hence the zero operator has multiple representations.
However, since the Berezin transform of $T$ is independent of
its representation, the Berezin transforms of $T^{\D_j}$ converge
to the same limit for any representation of $T$.
\end{remark}

For a function $\phi\in L^{q}(\D)$, assuming the Toeplitz operator $T_{\phi}$
is bounded on $A^2(\D)$,  we define the Berezin transform $B_{\D}\phi$
of $\phi$ as  $B_{\D}\phi(z)=B_{\D}T_{\phi}(z)$ for $z\in \D$. Hence
\begin{align*}
B_{\D}\phi(z) = \langle T_{\phi} k^{\D}_z,k^{\D}_z\rangle
    =\langle P_{\D}\phi k^{\D}_z,k^{\D}_z\rangle
    =\langle \phi k^{\D}_z,k^{\D}_z\rangle
    =\int_{\D}\phi(w)|k^{\D}_z(w)|^2dV(w).
\end{align*}
As a consequence of Theorem \ref{ThmLinfty} and Dini's Theorem
we have the following corollary.
\begin{corollary}\label{CorLInfty}
Let $\{\D_j\}$ be a sequence of domains in $\C^n$ such that
$\D_j\subset \D_{j+1}$ for all $j$ and
$\D=\cup_{j=1}^{\infty}\D_j$ be a non-trivial Bergman domain.
Assume that $\phi\in L^q(\D)$ for some $0<q<\infty$ so that
$T_{\phi}$ is bounded on $A^2(\D)$. Then there exists a subsequence
$\{j_k\}$ and functions $\phi_k\in L^{\infty}(\D_{j_k})$
such that $B_{\D_{j_k}}\phi_k\to B_{\D}\phi$ uniformly
on compact subsets of $\D\backslash \mathcal{Z}$. If $\D$ is
bounded, then $EB_{\D_{j_k}}\phi_k\to B_{\D}\phi$ in $L^p(\D)$
as $k\to \infty$ for all $0<p<\infty$.
\end{corollary}
We note that, as Proposition \ref{PropLp} below shows, $\phi_k$ in the
    corollary above might have to be different from $R_{\D_k}\phi$.
\begin{remark}
If the domain $\D$ is not bounded, then the Berezin transform
$B_{\D}T_{\phi}$ of the Toeplitz operator of a bounded symbol
$\phi$ does not have to be in $L^p(\D)$. For instance, let
$\phi(z)=\text{Re}(z)$ and  $\D=\{z\in\mathbb{C}: 0<\text{Re}(z)<1\}.$
We note that $K^{\D}(z,z)\neq 0$ for any $z\in \D$ as $(z+1)^{-1}$
    is square integrable on $\D$.
Since $\phi$ is bounded and harmonic, we conclude that
$B_{\D}T_{\phi}=\phi$ which is not in $L^p(\D)$ for any $0<p<\infty$.
\end{remark}

In the following proposition we compute the asymptotics of
the Berezin transform of $\log|z|$ on annuli that converge
to the punctured disc. Also it shows that the first conclusion in
Theorem \ref{ThmLinfty} is not true if we drop the assumption
that the symbol is bounded. The function
$\log|z|\in L^p(\mathbb{D}\setminus \{0\})$ for all
$0<p<\infty$ and, Lemma \ref{Lemma2} implies that,
\[B_{\mathbb{D}\setminus\{0\}}\log|z|
=B_{\mathbb{D}}\log|z|
=\frac{1}{2}(|z|^2-1).\]
\begin{proposition}\label{PropAsymp}
Let $A_r=\{z\in \C:r<|z|<1\}$ and $\phi(z)=\log|z|$. Then
\[B_{A_r}\phi(z)\to \frac{|z|^2}{4}-\frac{1}{4|z|^2}\]
uniformly on compact subsets of $\mathbb{D}\setminus\{0\}$
as $r\to 0^+$.
\end{proposition}

The following proposition shows that the last statement in
Theorem \ref{ThmLinfty} is not true in general for operators in
the Toeplitz algebra. One can argue as follows. Let 
$\phi(z)=\log|z|$ be a symbol on $\disk^*=\mathbb{D}\setminus \{0\}$. 
One can show that $T_{\phi}$ is compact on $A^2(\disk^*)$ 
(as $A^2(\disk^*)=A^2(\mathbb{D})$ and $\phi=0$ on the unit circle). 
However, compact operators are in the Toeplitz algebra 
(see \cite[Theorem 6]{Englis92}). Hence $T_{\phi}$ is in 
the Toeplitz algebra; yet, by Proposition \ref{PropLp} below,   
$\{B_{A_r}T^{A_r}_{\phi}\}$ does not converge to $B_{\disk^*}T_{\phi}$ in $L^p$. 

\begin{proposition}\label{PropLp}
Let $A_r=\{z\in \C:r<|z|<1\}$, $\disk^*=\disk\backslash \{0\}$, and $\phi(z)=\log|z|$.
Then  $T_{\phi}$ is a compact operator on $A^2(\disk^*)$ and
\begin{eqnarray*}
\lim_{r\to 0^+}\|EB_{A_r}T^{A_r}_{\phi}\|_{L^p(\disk^*)}= \infty,
\end{eqnarray*}
while $\|B_{\disk^*}T_{\phi}\|_{L^p(\disk^*)}<\infty$ for all $1\leq p\leq \infty$.
\end{proposition}

\section{Proofs of Theorems \ref{ThmBounded}, \ref{ThmDecreasing},
    \ref{ThmLinfty} and Corollary \ref{CorLInfty}}
We start with a simple lemma.
\begin{lemma}\label{Lem:Adjoint}
Let $\D$ be a non-trivial Bergman domain in $\C^n$ and $U\subset \D$
be a subdomain. Then $R_U^*K_z^U=K_z^{\D}$ for $z\in U$.
\end{lemma}

\begin{proof} For $z\in U$ and $f\in A^2(\D)$ we have
\[ f(z)=\langle R_Uf,K_z^U\rangle_U =\langle f,R_U^*K_z^U\rangle_{\D}.\]
Because of the uniqueness of the Bergman kernel, we conclude
that  $R_U^*K_z^U=K_z^{\D}$.
\end{proof}

 We will  need the following results of Ramadanov and Skwarczy\'nski
 (see \cite[Theorem 12.1.23 and Theorem 12.1.24]{JarnickiPflugBook1Ed2}
 and also \cite{Ramadanov67,Ramadanov83,IwinskiSkwarczynski75, SkwarczynskiThesis}).
 \begin{theorem}[Ramadanov]\label{ThmRamadanov}
Let $\D_j$ be an increasing sequence of domains in $\C^n$ such that
$\D=\cup_{j=1}^{\infty} \D_j$. Then, $K^{\D_j}\to K^{\D}$ as $j\to\infty$ locally
uniformly on $\D \times \D$.
\end{theorem}

 \begin{theorem}[Skwarczy\'nski]\label{ThmIwinskiSkwarczynski}
Let $\{\D_j\}$ be a sequence of domains in $\C^n$ such that
$\D\subset \D_{j+1}\subset \D_j$. Then, $K^{\D_j}\to K^{\D}$ as $j\to\infty$ locally
uniformly on $\D \times \D$ if and only if $K^{\D_j}(w,w)\to K^{\D}(w,w)$
as $j\to\infty$ for all $w\in \D$.
 \end{theorem}
Let $U$ be a subdomain of a domain $\D$. Since
\[K^{\D}(z,z)=\sup\{|f(z)|^2:f\in A^2(\D)\text{ and } \|f\|=1\},\]
we have $0\leq K^{\D}(z,z) \leq K^{U}(z,z)$ for every $z\in U$. Hence, if
$K^{\D}(z,z)\not = 0$, then $K^{U}(z,z)\not = 0$.

\begin{lemma}\label{LemRestriction}
Let $\{\D_j\}$ be a sequence of domains in $\C^n$ such that
$\D_j\subset \D_{j+1}$ for all $j$ and
$\D=\cup_{j=1}^{\infty}\D_j$ be a non-trivial Bergman domain.
Then for each compact set $K\subset\D\backslash \mathcal{Z}$, we have
\[\lim_{j\to\infty}\sup_{z\in K}\|R^*_{\D_j}k_z^{\D_j}-k^{\D}_z\|_{L^2(\D)}=0.\]
\end{lemma}

\begin{proof}
First we note that $0\leq K^{\D}(z,z)\leq K^{\D_j}(z,z)$ for all $j$ and $z\in K$.
So since $K\subset \D\setminus \mathcal{Z}$ we have $K^{\D_j}(z,z)\neq 0$ for
all $j$ so that $K\subset \D_j$. Let $j_0$ be chosen such
that $K\subset \D_{j_0}$.  Lemma \ref{Lem:Adjoint} implies that
$R^*_{\D_j}k_z^{\D_j}=K_z^{\D}/\sqrt{K^{\D_j}(z,z)}$ for $j\geq j_0$.
Then for $z\in K$ and $j\geq j_0$ we have
\begin{align*}
\|R^*_{\D_j}k_z^{\D_j}- k_z^{\D}\|_{L^2(\D)}
=& \left\|\frac{K_z^{\D}}{\sqrt{K^{\D_j}(z,z)}}
-\frac{K_z^{\D}}{\sqrt{K^{\D}(z,z)}}\right\|_{L^2(\D)} \\
=&\left\|k_z^{\D}\left(1-\sqrt{K^{\D}(z,z)}/\sqrt{K^{\D_j}(z,z)}\right)\right\|_{L^2(\D)} \\
= &\left|1-\sqrt{K^{\D}(z,z)}/\sqrt{K^{\D_j}(z,z)}\right|.
\end{align*}
Ramadanov's Theorem (Theorem \ref{ThmRamadanov}) implies that
$K^{\D}(z,z)/K^{\D_j}(z,z)\to 1$ uniformly on $K$ as $j\to \infty$.
Therefore,
$\sup_{z\in K}\|R^*_{\D_j}k_z^{\D_j}-k^{\D}_z\|_{L^2(\D)}\to 0$ as $j\to \infty$.
\end{proof}

The following Lemma, which is used in the proof of Theorem \ref{ThmBounded},
might be of interest on its own right.

\begin{lemma} \label{LemSeries}
Let $\D$ be a non-trivial Bergman domain in $\C^n$ and $U\subset \D$
be a subdomain. Let $T$ be a bounded operator on $A^2(\D)$. Then
\[\frac{B_{\D}T(z)}{B_U(R_UTR^*_U)(z)} = \frac{K^U(z,z)}{K^{\D}(z,z)}\]
for $z\in U\backslash \mathcal{Z}$.
\end{lemma}
\begin{proof}
For $z\in U\backslash \mathcal{Z}$, we use Lemma \ref{Lem:Adjoint} to get
\begin{align*}
B_U(R_UTR_U^*)(z)
=&\langle TR_U^*k_z^U,R_U^*k_z^U\rangle_{\D} \\
=&\frac{\langle TK_z^{\D},K_z^{\D}\rangle_{\D}}{K^U(z,z)} \\
=&\frac{K^{\D}(z,z)}{K^U(z,z)}B_{\D}T(z).
\end{align*}
Hence the proof of Lemma \ref{LemSeries} is complete.
\end{proof}

\begin{corollary}
Let $\D$ be a non-trivial Bergman domain in $\C^n, U\subset \D$
be a subdomain, and $T$ be a bounded linear
operator on $A^2(\D)$. Assume that $p\in \overline{U}$ and $1\leq\alpha<\infty$
such that $\frac{K^U(z,z)}{K^{\D}(z,z)}\to \alpha$ as $z\to p$, $z\in U\backslash \mathcal{Z}$. Then
$B_{\D}T$ is continuous at $p$ if and only if $B_U(R_UTR^*_U)$ is
continuous at $p$.
\end{corollary}

Now we are ready to prove Theorem \ref{ThmBounded}.

\begin{proof}[Proof of Theorem \ref{ThmBounded}]
The proof of locally uniform convergence is a result of Theorem
\ref{ThmRamadanov} together with Lemma \ref{LemSeries}. Indeed,
Theorem \ref{ThmRamadanov} implies that
\begin{eqnarray*}
K^{\D_j}(z,z)/K^{\D}(z,z)\to 1
\end{eqnarray*}
locally uniformly on $\D\times \D$ as $j\to \infty$. Then Lemma \ref{LemSeries}
implies that
\begin{eqnarray*}
B_{\D_j}R_{\D_j}TR^*_{\D_j}\to B_{\D}T
\end{eqnarray*}
locally uniformly on
$\D$ as $j\to\infty$.

To prove the second part we assume that $\D$ is bounded and  $0<p<\infty$.
From the first part of the proof, we know that
$B_{\D_j}R_{\D_j}TR^*_{\D_j}\to B_{\D}T$ uniformly on compact sets as
$j\to \infty$. Furthermore,
$|B_{\D}T(z)|\leq \|T\|$ and $|EB_{\D_j}R_{\D_j}TR^*_{\D_j}(z)|\leq \|T\|$
for all $z\in \D$ and all $j$. Then, using the Lebesgue Dominated
Convergence Theorem, we conclude that
$EB_{\D_j}R_{\D_j}TR^*_{\D_j}\to B_{\D}T$ in $L^p(\D)$  as $j\to\infty$.
\end{proof}

\begin{lemma}\label{LemL2Conv}
Let $\D$ be a non-trivial Bergman domain and $\{\D_j\}$ be a sequence of
domains in $\C^n$ such that $\D\subset\D_{j+1}\subset \D_j$ for all $j$.
Assume that $K_{\D_j}(z,z)\to K_{\D}(z,z)$ as $j\to\infty$ for every $z\in\D$.
Then for each compact set $K\subset\D\backslash\mathcal{Z}$, we have
\[\lim_{j\to\infty}\sup_{z\in K}\|R^{\D_j}_{\D}k^{\D_{j}}_z-k^{\D}_z\|_{L^2(\D)}=0.\]
\end{lemma}

\begin{proof}
If $K^{\D}(z,z)>0$ for some $z\in\D$, then $K^{\D_j}(z,z)>0$ for large $j$
because $K^{\D_j}(z,z)$ increases to $K^{\D}(z,z)$ as $j\to\infty$. Furthermore,
there exists an open neighborhood of $z$ for which the normalized 
Bergman kernels $k^{\D_{j}}$ and $k^{\D}$ are well-defined for $j$ 
large enough. Since $K\subset\D\backslash \mathcal{Z}$ is compact, 
all of the functions in the statement are well-defined for large $j$, 
and the limit makes sense.

Let $\ep>0$ be given.  For each $z\in K$, we choose a compact
$S_{z}\subset \D$ so that
$\left\| k^{\D}_z\right\|_{L^2(\D \backslash S_z)}<\ep$.
Recall that the map $z\mapsto k^{\D}_{z}$ is continuous from
$\D\backslash \mathcal{Z}$ to $L^2(\D)$ (see \eqref{Eqn0}). For any
$z\in \D\setminus \mathcal{Z}$ we choose an open set
$U_z\subset \D\backslash \mathcal{Z}$ so that $z\in U_z$
and $\left\| k^{\D}_z-k^{\D}_w\right\|_{L^2(\D)}<\ep$
when $w\in U_{z}$. Then
\[\left\| k^{\D}_w\right\|_{L^{2}(\D \backslash S_z)}
<\ep+\left\| k^{\D}_z\right\|_{L^2(\D\backslash S_z)}
<2\ep\]
for $w\in U_{z}$. Since $K$ is compact, there exist
$z_{1},\cdots ,z_{m}\in K$ so that $K\subset \cup _{j=1}^{m}U_{z_j}$.
The set $S=\cup _{j=1}^{m}S_{z_j}\subset \D$ is compact as well and
\begin{eqnarray*}
\sup_{w\in K}\left\| k^{\D}_w \right\|_{L^2(\D \backslash S)}<2\ep.
\end{eqnarray*}
Using Theorem \ref{ThmIwinskiSkwarczynski}, we have
\begin{align}\label{Eqn1}
\sup_{z\in K, w\in S}\left |k^{\D_j}_z(w)-k^{\D}_z(w)\right |
<\frac{\ep}{\sqrt{Vol(S)+1}}
\end{align}
and
\begin{align*}
\sup_{z\in K, w\in S}\left ||k^{\D_j}_z(w)|^2-|k^{\D}_z(w)|^2\right|
<\frac{\ep^2}{Vol(S)+1}
\end{align*}
for large enough $j$. Then by integrating the above inequality
over $S$ and using
$\left\| k^{\D}_z\right\|_{L^2(\D \backslash S)}<2\ep$
we get
\begin{align*}
\|k_z^{\D_j}\|^2_{L^2(S)}\geq \|k_z^{\D}\|^2_{L^2(S)}-\ep^2
> 1-4\ep^2-\ep^2 =1-5\ep^2,
\end{align*}
which implies that $\|k_z^{\D_j}\|_{L^2(\D\backslash S)}<\sqrt{5}\ep$
when $j$ is large enough. Then using \eqref{Eqn1} we get
\begin{align*}
\|R^{\D_j}_{\D}k^{\D_{j}}_z-k^{\D}_z\|_{L^2(\D)}
&\leq \left\| k^{\D_{j}}_z-k^{\D}_z  \right\|_{L^2(S)}
+\|k^{\D}_z\|_{L^2(\D\backslash S)}+\|k^{\D_j}_z\|_{L^2(\D\backslash S)} \\
&<(3+\sqrt{5})\ep
\end{align*}
for $j$ large and $z\in K$. Hence,
\[\lim_{j\to\infty}\sup_{z\in K}\|R^{\D_j}_{\D}k^{\D_{j}}_z-k^{\D}_z\|_{L^2(\D)}=0.\]
The proof is finished.
\end{proof}

\begin{proof}[Proof of Theorem \ref{ThmDecreasing}]
For $z\in\D\backslash \mathcal{Z}$, we define $f(z)=B_{\D}T (z)$ and
\begin{align*}
f_j(z)=&B_{\D_j}(R^{\D_j}_{\D})^*TR^{\D_j}_{\D}(z)\\
g_j(z)=&\langle TR^{\D_j}_{\D}k^{\D_{j}}_z,k^{\D}_z \rangle_{L^2(\D)}
\end{align*}
for each $j$. Then
\[f_j(z)
=\langle (R^{\D_j}_{\D})^*TR^{\D_j}_{\D}k^{\D_{j}}_z,k^{\D_j}_z \rangle _{L^2(\D_j)}
= \langle TR^{\D_j}_{\D}k^{\D_{j}}_z,R^{\D_j}_{\D}k^{\D_j}_z \rangle_{L^2(\D)}.\]
Let $K\subset\D$ be a compact set. By Cauchy-Schwarz inequality we have
\begin{align*}
\sup_{z\in K}|g_j(z)-f(z)|
=&\sup_{z\in K}\left| \langle TR_{\D}^{\D_j}k_z^{\D_j}-Tk_z^{\D},
k_z^{\D}\rangle \right| \\
\leq& \sup_{z\in K}\left\| TR_{\D}^{\D_j}k_z^{\D_j}-Tk_z^{\D}\right\|_{L^2(\D)}\\
\leq &\left\| T\right\| \sup_{z\in K}
\left\| R_{\D}^{\D_j}k_z^{\D_j}-k_z^{\D}\right\|_{L^2(\D)}.
\end{align*}
The last term above converges to zero by Lemma \ref{LemL2Conv}.
Therefore, the sequence $\{g_j\}$ converges to $f$ uniformly on $K$.

Using Cauchy-Schwarz inequality again we have
\begin{align*}
|f_j(z)-g_j(z)| = \left|\langle  TR^{\D_j}_{\D}k^{\D_j}_z,
    R^{\D_j}_{\D}k^{\D_j}_z-k^{\D}_z \rangle \right|
    \leq  \|T\|\|R^{\D_j}_{\D}k^{\D_j}_z-k^{\D}_z\|_{L^2(\D)}.
\end{align*}
Lemma \ref{LemL2Conv} implies that the last term above converges
to zero uniformly  on $K$. Hence, $|f_j-g_j|\to 0$ uniformly on
$K$ as $j\to \infty$. Therefore, $\{f_j\}$ converges to $f$ uniformly on $K$. 

As in the proof of Theorem \ref{ThmBounded} we prove the second part 
as follows. We assume that $\D$ is bounded. From the previous part of 
this proof we know that $\{f_j\}$ converges to $f$ uniformly on compact 
subset of $\D$. Furthermore, $\|f_j\|_{L^{\infty}(\D)}\leq \|T\|$ for all $j$. 
Then using the  Lebesgue Dominated Convergence Theorem, we conclude 
that $\{f_j\}$ converges to $f$ in $L^p(\D)$  as $j\to\infty$ for all $0<p<\infty$.
\end{proof}

Now we are ready to prove Theorem \ref{ThmLinfty}.

\begin{proof}[Proof of Theorem \ref{ThmLinfty}]
It is enough to prove the result for finite product of Toeplitz
operators as it is easy to conclude the theorem for the finite sums
of such operators. So let $T=T_{\phi_m}\cdots T_{\phi_1}$ where
$\phi_1, \ldots, \phi_m\in L^{\infty}(\D)$. One can easily show
that $B_{\D}T \in  L^{\infty}(\D)$ and
$B_{\D_j}T^{\D_j}\in  L^{\infty}(\D_j)$ for all $j$. Furthermore,
one can show that
\[\max\{\|B_{\D_j}T^{\D_j}\|_{L^{\infty}(\D_j)},
 \|B_{\D}T\|_{L^{\infty}(\D)}\}
\leq \|\phi_1\|_{L^{\infty}(\D)}\cdots \|\phi_m\|_{L^{\infty}(\D)}.\]
Let $f_j(z)=|B_{\D}T(z)-EB_{\D_j}T^{\D_j}(z)|$ for $z\in \D$. Then
\begin{align}\label{IneqBounded}
\|f_j\|_{L^{\infty}(\D)}
    \leq 2\|\phi_1\|_{L^{\infty}(\D)}\cdots \|\phi_m\|_{L^{\infty}(\D)}
\end{align}
for all  $j$.

We will use induction to prove that
\[\sup\{|T^{\D_j}k_z^{\D_j}(w)-Tk_z^{\D}(w)|:z,w\in K\}\to 0 \]
as $j\to \infty$.
So first let us assume that
$T=T_{\phi_1}$ is a Toeplitz operator. Let $K$ be a compact
set in $\D\backslash\mathcal{Z}$. As in the proof of Lemma \ref{LemL2Conv}
for a given $\ep>0$, there exists  a compact set $S\subset \D$ and
$j_0\in\mathbb{N}$ such that $K\Subset \D_j$,
$\|k^{\D}_z\|_{L^2(\D\setminus S)}<\ep$ for all $z\in K$,
and $\|k^{\D_j}_z\|_{L^2(\D_j\setminus S)}<\ep$ for all $z\in K$
and $j\geq j_0$. Let us consider the following equalities.

\begin{align*}
T_{\phi_1}k^{\D}_z(w)-T^{\D_j}_{\phi_1}k^{\D_j}_z(w)
=& \langle\phi_1 k^{\D}_z,K^{\D}(.,w)\rangle_{\D}
-\langle \phi_1 k^{\D_j}_z,K^{\D_j}(.,w)\rangle_{\D_j}\\
=&\langle\phi_1 k^{\D}_z,K^{\D}(.,w)\rangle_{S}
-\langle \phi_1 k^{\D_j}_z,K^{\D_j}(.,w)\rangle_{S}\\
&+\langle\phi_1 k^{\D}_z,K^{\D}(.,w)\rangle_{\D\setminus S}
-\langle \phi_1 k^{\D_j}_z,K^{\D_j}(.,w)\rangle_{\D_j\setminus S}.
\end{align*}
There exists $C_K>1$ such that $1/C_K\leq K^{\D_j}(w,w)\leq C_K$ for all $w\in K$
and all $j\geq j_0$ since by Theorem \ref{ThmRamadanov}, the continuous
functions $\{K^{\D_j}(w,w)\}$ converges to $K^{\D}(w,w)$ uniformly on $K$.

Without loss of generality we can assume that
\begin{align*}
\|k^{\D}_z\|_{L^2(\D\setminus S)}
&<\frac{\ep}{\sqrt{K^{\D}(w,w)}}, \\
\|k^{\D_j}_z\|_{L^2(\D_j\setminus S)}
&<\frac{\ep}{\sqrt{K^{\D_j}(w,w)}}
\end{align*}
for $j\geq j_0$ and all $z,w\in K$. Then
\[\left|\langle\phi_1 k^{\D}_z,K^{\D}(.,w)\rangle_{\D\setminus S}\right|
+ \left|\langle \phi_1 k^{\D_j}_z,K^{\D_j}(.,w)\rangle_{\D_j\setminus S}\right|
\leq 2\ep \|\phi_1\|_{L^{\infty}(\D)}\]
for all $z,w\in K$.  Also
\[\sup\left\{\left|\langle \phi_1 k^{\D}_z,K^{\D}(.,w)\rangle_{S}
- \langle\phi_1 k^{\D_j}_z,K^{\D_j}(.,w)\rangle_S\right|:z,w\in K\right\}\to 0\]
as $j\to \infty$ (a consequence of Theorem \ref{ThmRamadanov}). Then
\[\limsup_{j\to\infty} \sup\left\{\left|T_{\phi_1}k^{\D}_z(w)
-T^{\D_j}_{\phi_1}k^{\D_j}_z(w)\right|:z,w\in K\right\}
\leq 2\ep \|\phi_1\|_{L^{\infty}(\D)}.\]
Since $\ep$ is arbitrary, we conclude that
\[\sup\left\{\left|T^{\D}_{\phi_1}k^{\D}_z(w)
-T^{\D_j}_{\phi_1}k^{\D_j}_z(w)\right|:z,w\in K\right\}\to 0\]
as $j\to \infty$. We note that for $z\in \D_j$ we have
\begin{align}
\nonumber \left|B_{\D_j}T_{\phi_1}^{\D_j}(z)\right.-&\left.B_{\D}T_{\phi_1}(z)\right|
= \frac{\left| \sqrt{\frac{K^{\D}(z,z)}{K^{\D_j}(z,z)}}
\left\langle T^{\D_j}_{\phi_1}k_{z}^{\D_j},K_z^{\D_j}\right\rangle_{\D_j}
-\left\langle T_{\phi_1}k_z^{\D},K_z^{\D}\right\rangle_{\D }\right|}{\sqrt{K^{\D}(z,z)}} \\
\label{Eqn4}=&\frac{1}{\sqrt{K^{\D}(z,z)}}\left|\sqrt{\frac{K^{\D}(z,z)}{K^{\D_j}(z,z)}}
T^{\D_j}_{\phi_1}k_z^{\D_j}(z)-T_{\phi_1}k_{z}^{\D}(z)\right|.
\end{align}
Hence $B_{\D_j}T^{\D_j}_{\phi_1}\to B_{\D}T_{\phi_1}$ uniformly on compact
subsets of $\D\backslash \mathcal{Z}$ as $j\to \infty$.

Next we show the induction step. Let $\widetilde{T}=T_{\phi_{m-1}}\cdots T_{\phi_1}$
and $\widetilde{T}^{\D_j}=T_{\phi_{m-1}|_{\D_j}}\cdots T_{\phi_1|_{\D_j}}$.
As the induction hypothesis we assume that
$ \widetilde{T}^{\D_j}k^{\D_j}_z \to \widetilde{T}k^{\D}_z $ uniformly on
compact subsets  as $j\to \infty$. Then
\begin{align*}
Tk^{\D}_z(w)-T^{\D_j}k^{\D_j}_z(w)
=& \langle\phi_m\widetilde{T}k_z^{\D}, K^{\D}(.,w)\rangle_{\D}
- \langle\phi_m\widetilde{T}^{\D_j}k_z^{\D_j}, K^{\D_j}(.,w)\rangle_{\D_j}\\
=&\langle\phi_m\widetilde{T}k_z^{\D}, K^{\D}(.,w)\rangle_{S}
- \langle\phi_m\widetilde{T}^{\D_j}k_z^{\D_j}, K^{\D_j}(.,w)\rangle_{S}\\
&+\langle\phi_m\widetilde{T}k_z^{\D}, K^{\D}(.,w)\rangle_{\D\setminus S} \\
&- \langle\phi_m\widetilde{T}^{\D_j}k_z^{\D_j}, K^{\D_j}(.,w)\rangle_{\D_j\setminus S}
\end{align*}
As in the previous case, we have
\begin{align*}
\left|\langle\phi_m\widetilde{T}k_z^{\D}, K^{\D}(.,w)\rangle_{\D\setminus S}\right|
\leq & \|\phi_m\|_{L^{\infty}(\D)}\|\widetilde{T}\|\|k_z^{\D}\|_{L^2(\D\setminus S)}\sqrt{K^{\D}(w,w)}\\
\leq &\ep \|\phi_m\|_{L^{\infty}(\D)}\cdots  \|\phi_1\|_{L^{\infty}(\D)}
\end{align*}
and
\begin{align*}
\left|\langle\phi_m\widetilde{T}^{\D_j}k_z^{\D_j},
K^{\D_j}(.,w)\rangle_{\D_j\setminus S}\right|
\leq &\|\phi_m\|_{L^{\infty}(\D)}\|\widetilde{T}^{\D_j}\|\|k_z^{\D_j}\|_{L^2(\D_j\setminus S)}\sqrt{K^{\D_j}(w,w)}\\
\leq &\ep \|\phi_m\|_{L^{\infty}(\D)}\cdots  \|\phi_1\|_{L^{\infty}(\D)}.
\end{align*}
Then
\begin{align*}
\left|\langle\phi_m\widetilde{T}k_z^{\D},
K^{\D}(.,w)\rangle_{\D\setminus S}\right|
&+\left|\langle\phi_m\widetilde{T}^{\D_j}k_z^{\D_j},
K^{\D_j}(.,w)\rangle_{\D_j\setminus S}\right|\\
&\leq 2\ep \|\phi_m\|_{L^{\infty}(\D)}\cdots  \|\phi_1\|_{L^{\infty}(\D)}
\end{align*}
for all $z,w\in K$. Furthermore, by  induction hypothesis, we have
\[\sup\{|\widetilde{T}^{\D_j}k_z^{\D_j}(w)- \widetilde{T}k_z^{\D}(w)|:z,w\in K\}
\to 0\]
 as $j\to \infty$. Then
\[\sup\left\{\langle\phi_m\widetilde{T}k_z^{\D}, K^{\D}(.,w)\rangle_{S}
- \langle\phi_m\widetilde{T}^{\D_j}k_z^{\D_j},
    K^{\D_j}(.,w)\rangle_{S}:z,w\in K\right\}\to 0\]
as $j\to \infty$. Hence,
\[\sup\{|T^{\D_j}k_z^{\D_j}(w)- Tk_z^{\D}(w)|:z,w\in K\} \to 0\]
as $j\to \infty$. Similar to \eqref{Eqn4} one can show that
\begin{align*}\label{Eqn5}
\left|B_{\D_j}T^{\D_j}(z)-B_{\D}T(z)\right|
=\frac{1}{\sqrt{K^{\D}(z,z)}}\left|\sqrt{\frac{K^{\D}(z,z)}{K^{\D_j}(z,z)}}
T^{\D_j}k_z^{\D_j}(z)-Tk_{z}^{\D}(z)\right|.
\end{align*}
Therefore, $f_j\to 0$ uniformly on $K$ as $j\to\infty$.

To prove the second part we assume that $\D$ is bounded. Then 
the Lebesgue Dominated Convergence Theorem together 
with \eqref{IneqBounded} implies that $\int_{\D}|f_j(z)|^pdV(z)\to 0$ 
as $j\to \infty$. Hence, $EB_{\D_j}T^{\D_j} \to B_{\D}T$ in $L^p(\D)$ 
as $j\to \infty$.
\end{proof}

Using very similar arguments as in the proof of Theorem \ref{ThmLinfty}
one can prove the following corollary.
\begin{corollary}\label{CorLinftyDecreasing}
Let $\D$ be a non-trivial Bergman domain and $\{\D_j\}$ be a sequence of
domains in $\C^n$ such that $\D\subset\D_{j+1}\subset \D_j$ for all $j$.
Assume $K^{\D_j}(z,z)\to K^{\D}(z,z)$ as $j\to\infty$ for every $z\in\D$.
Let  $T=\sum_{m=1}^lT_{\phi_{m,1}}\cdots T_{\phi_{m,k_m}}$
be a finite sum of finite products of Toeplitz operators with
bounded symbols on $\D_1$ and
$T^{\D_j} =\sum_{m=1}^lT_{\phi_{m,1}|_{\D_j}}\cdots T_{\phi_{m,k_m}|_{\D_j}}$
for each $j$. Then $B_{\D_j}T^{\D_j}\to B_{\D}T$ uniformly on compact
subsets of $\D\setminus \mathcal{Z}$  as $j\to \infty$.
Furthermore, if $\D_1$ is bounded, then $EB_{\D_j}T^{\D_j}\to B_{\D}T$
in $L^p(\D)$ as $j\to \infty$ for all $0<p<\infty$.
\end{corollary}

We finish this section with the proof of Corollary \ref{CorLInfty}.

\begin{proof}[Proof of Corollary \ref{CorLInfty}]
Let $\phi\in L^q(\D)$ and let $K\subset\D\backslash\mathcal{Z}$ be compact.
First assume that $\phi$ is real valued and $\phi\geq 0$ on $\D$.
For each $k\geq 1$ we define $\phi_k=\min\{\phi, k\}$.
Hence, $\phi_k\in L^{\infty}(\D)$ and $B_{\D}\phi_k(z)$ increases to
$B_{\D}\phi(z)$ for each $z\in\D$. By Dini's Theorem, $B_{\D}\phi_k$
converges uniformly to $B_{\D}\phi$ on $K$. By Theorem \ref{ThmLinfty},
for each $k\geq 1$ there exists $j_k$ so that
\[\sup_{z\in K}|EB_{\D_{j_k}}\phi_k(z)-B_{\D}\phi_k(z)|\leq \frac{1}{k}.\]
This means that $EB_{\D_{j_k}}\phi_k$ converges uniformly to
$B_{\D}\phi$ on $K$. If $\D$ is bounded and $p>0$, then by the
last statement of Theorem \ref{ThmLinfty}, we can find $j_k$ so that
$\|EB_{\D_{j_k}}\phi_k-B_{\D}\phi_k\|_{L^p(\D)}\leq 1/k$.
By Monotone Convergence Theorem, we conclude that
$\|B_{\D}\phi_k-B_{\D}\phi\|_{L^p(\D)}\to 0$ as $k\to\infty$.
Therefore, $\|EB_{\D_{j_k}}\phi_k-B_{\D}\phi\|_{L^p(\D)}\to 0$ as
$k\to\infty$. Now let $\phi\in L^q(\D)$ be real valued. Then we
write $\phi=\phi^+-\phi^-$ where $\phi^+, \phi^-\geq 0$ on $\D$.
Since $B_{\D}\phi=B_{\D}\phi^+-B_{\D}\phi^-$, we can apply the 
first part of the proof to each term. Finally, if $\phi$ is complex 
valued then we can apply the previous part of the proof to the 
real and imaginary parts of $\phi$.
\end{proof}

\section{Proofs of Propositions \ref{PropAsymp} and \ref{PropLp}}

Let $\disk=\{z\in\C:|z|<1\}$ be the unit disk in the complex plane.
The Poisson kernel (see, for instance, \cite[Definition 1.2.3]{RansfordBook}) 
on the unit disk is defined as 
 \[P(z,\zeta)=\text{Re}\left(\frac{\zeta+z}{\zeta-z}\right) 
=\frac{1-|z|^2}{|\zeta-z|^2}\]
where $z\in\disk$, $|\zeta|=1$. 
\begin{lemma}\label{Lemma1}
Let  $0<s<1$ and $z\in\disk$. Then
\begin{align*}
\frac{1}{2\pi}\int_0^{2\pi} (P(sz,e^{it}))^2dt=\frac{1+s^2|z|^2}{1-s^2|z|^2}.
\end{align*}
\end{lemma}

\begin{proof}
Let us fix $z=\rho e^{i\theta}$. In \eqref{Eqn2},
we use the property that
\[P(s\rho e^{i\theta},e^{it})=P(s\rho e^{it},e^{i\theta});\]
and in \eqref{Eqn3} we use the facts that $P$, the Poisson
kernel, is the kernel of the integral operator that solves the
Dirichlet problem and $P(.,e^{it})$ is harmonic on $\mathbb{D}$
(see \cite{RansfordBook}).
\begin{align}
\nonumber \frac{1}{2\pi}\int_0^{2\pi} \left(P(sz,e^{it})\right)^2dt=&
\frac{1}{2\pi}\int_0^{2\pi} P(sz,e^{it})P(sz,e^{it})dt\\
\nonumber =&\frac{1}{2\pi}\int_0^{2\pi} P(s\rho e^{i\theta},e^{it})P(sz,e^{it})dt\\
\label{Eqn2}=& \frac{1}{2\pi}\int_0^{2\pi} P(s\rho e^{it},e^{i\theta})P(sz,e^{it})dt\\
\label{Eqn3} =& P(s^2\rho z,e^{i\theta})\\
\nonumber =&\frac{1-s^4|z|^4}{(1-s^2|z|^2)^2}\\
\nonumber =&\frac{1+s^2|z|^2}{1-s^2|z|^2}.
\end{align}
Hence, the proof of Lemma \ref{Lemma1} is complete.
\end{proof}

A function $u(z,w)$ in $\disk^2$ is said to be separately subharmonic if
when one of the variables is fixed in $\disk$, $u$ is subharmonic in the
other variable.

\begin{lemma}\label{Lemma2}
Let $\displaystyle G_a(z)=\log\left | \frac{a-z}{1-\ab z}\right | $ be the Green's
function for $\mathbb{D}$ with pole at $a\in \mathbb{D}$. Then
\begin{align*}
B_{\mathbb{D}}G_a(z)=\frac{1}{2}\left (\left |\frac{a-z}{1-\ab z}
\right|^2-1\right)
\end{align*}
and  the function $u(z,a)=B_{\mathbb{D}}G_a(z)$, defined for
$(z,a)\in\disk ^2$, is separately subharmonic on $\disk ^2$.
\end{lemma}

\begin{proof}
First suppose that $a=0$. Using Lemma \ref{Lemma1} in the fourth
equality below  we get
\begin{align*}
B_{\mathbb{D}}G_0(z)=&\frac{(1-|z|^2)^2}{\pi}
\int_{\disk}\frac{\log|w|}{|1-\wb z|^4}dV(w)\\
=&\frac{(1-|z|^2)^2}{\pi}\int_0^1 s\log s \int_0^{2\pi}
\frac{1}{|1-se^{-it} z|^4}dtds\\
=&2 (1-|z|^2)^2\int_0^1 \frac{s\log s}{(1-s^2|z|^2)^2}
\frac{1}{2\pi}\int_0^{2\pi}
\frac{(1-s^2|z|^2)}{|e^{it}-sz|^2}
\frac{(1-s^2|z|^2)}{|e^{it}-sz|^2}dt
ds\\
=&2 (1-|z|^2)^2\int_0^1 \frac{s\log s}{(1-s^2|z|^2)^2}
\frac{1+s^2|z|^2}{(1-s^2|z|^2)}ds\\
=&2 (1-|z|^2)^2\int_0^1
\frac{s(1+s^2|z|^2)\log s}{(1-s^2|z|^2)^3}ds.
\end{align*}
One can show that
\[\int \frac{x(1+|z|^2x^2)\log x}{(1-|z|^2x^2)^3}dx\
=\frac{x^2\log x}{2(|z|^2x^2-1)^2}+\frac{1}{4|z|^2(|z|^2x^2-1)}+C.\]
Therefore,
\[2 (1-|z|^2)^2\int_0^1
\frac{s(1+s^2|z|^2)\log s}{(1-s^2|z|^2)^3}ds
=\frac{1}{2}(|z|^2-1).\]

Let $a\in\disk\backslash\{0\}$. Let $\displaystyle \psi_a(w)=\frac{a-w}{1-\ab w}$
be the M\"{o}bius transform on the disk. Then, using
 \cite[Chapter 2]{HedenmalmKorenblumZhuBook}
 (see also  \cite[Section 6.3]{ZhuBook}) we have
\begin{align*}
B_{\mathbb{D}}G_a(z)=&\int_{\disk}G_a(\psi_z(w))dV(w)\\
=&\int_{\disk}G_0(\psi_a\circ\psi_z(w))dV(w) \\
=&\int_{\disk}G_0(\psi_{\psi_{a}(z)}(w))dV(w)\\
=& B_{\mathbb{D}}G_0(\psi_{a}(z)) = \frac{1}{2}\left (\left
|\frac{a-z}{1-\ab z}\right |^2-1\right ).
\end{align*}
Hence, the proof of Lemma \ref{Lemma2} is complete.
\end{proof}

\begin{proof}[Proof of Proposition \ref{PropAsymp}]
The Bergman kernel of the annulus  $A_r$ is
(see \cite[Example 12.1.7 (c)]{JarnickiPflugBook1Ed2})
\[K^{A_r}(z,w)=-\frac{1}{2\pi z\wb \log r}
+\frac{1}{\pi z\wb } \sum_{k\neq 0} \frac{kz^k\wb^k}{1-r^{2k}}.\]
Let $K$ be a compact subset of $\mathbb{D}\setminus\{0\}$. Then
for small enough $r>0$ the set $K$ is a compact subset of $A_r$.
Let us fix $z_0\in K\Subset  A_r$ and let us
break down the function $K^{A_r}(z_0,w)$ into four pieces as
\[K^{A_r}(z_0,w)=\psi^0_{r,z_0}(w)+\psi^1_{r,z_0}(w)
+\psi^2_{r,z_0}(w) +\psi^3_{r,z_0}(w)\]
where
\begin{align*}
 \psi^0_{r,z_0}(w)=&-\frac{1}{2\pi z_0\wb \log r},\\
 \psi^1_{r,z_0}(w)=&\frac{r^2}{(1-r^2)\pi z_0^2 \wb^2},\\
 \psi^2_{r,z_0}(w)=& \frac{1}{\pi z_0^2}
\sum_{k=2}^{\infty}\frac{k}{1-r^{2k}}
\left(\frac{r}{z_0}\right)^{k-1}\left(\frac{r}{\wb}\right)^{k+1},\\
\psi^3_{r,z_0}(w)=&\frac{1}{\pi z_0\wb } \sum_{k=1}^{\infty}
\frac{kz_0^k\wb^k}{1-r^{2k}}.
\end{align*}
One can check that the $\sup\{|\psi^1_{r,z_0}(w)|:z_0\in K,w\in A_r\}$
and $\sup\{|\psi^3_{r,z_0}(w)|:z_0\in K,w\in A_r\}$ stay bounded  as
$r\to 0^+$. Furthermore, $\sup\{|\psi^2_{r,z_0}(w)|:z_0\in K,w\in A_r\}$
converges to zero as $r\to 0^+$.

Now we will estimate the Berezin  transform  of $\phi(w)=\log|w|$
on $A_r$ at $z_0$. First we can write $|K^{A_r}(z_0,w)|^2$ as
\begin{align*}
|K^{A_r}(z_0,w)|^2 =\left|\psi^0_{r,z_0}(w)\right|^2
+\left|\psi^3_{r,z_0}(w)\right|^2 +\Psi_{r,z_0}(w)
\end{align*}
where
\begin{align*}
\Psi_{r,z_0}(w)=&2\text{Re}\left(\psi^0_{r,z_0}(w)\sum_{j=1}^3
\overline{\psi^j_{r,z_0 }(w)}
+\psi^1_{r,z_0}(w)\sum_{j=2}^3\overline{\psi^j_{r,z_0 } (w)}\right)\\
&+2\text{Re}\left(\psi^2_{r,z_0}(w)\overline{\psi^3_{r,z_0}(w)}\right)
+\left|\psi^1_{r,z_0}(w)\right|^2+\left|\psi^2_{r,z_0}(w)\right|^2.
\end{align*}

Now we will show that
$\sup\left\{\left|\int_{A_r}\phi(w)\Psi_{r,z_0}(w)dV(w)\right|
:z_0\in K\right\}\to 0$ as $r\to 0^+$.
Using polar coordinates we compute
\begin{align*}
\int_{A_r}|\phi(w)|\left|\psi^0_{r,z_0}(w)\right| dV(w)
=&\frac{1}{|z_0|\log r}\int_r^1 \log \rho d\rho \\
=&\frac{r-r\log r-1}{|z_0|\log r}\to 0
\end{align*}
uniformly on $K$ as $r\to 0^+$. Hence using the fact that
$\psi^1_{r,z_0},\psi^2_{r,z_0},\psi^3_{r,z_0}$ stay bounded uniformly
on $A_r$ for all $z_0\in K$  we conclude that
\[\int_{A_r}\phi(w)\psi^0_{r,z_0}(w)
\sum_{j=1}^3\overline{\psi^j_{r,z_0 } (w)} dV(w) \to 0\]
uniformly on $K$ as $r\to 0^+$.  Similarly,  we conclude that
\[\int_{A_r}\phi(w) \left|\psi^1_{r,z_0}(w)\right|^2 dV(w)\to 0\] 
 and 
 \[ \int_{A_r}\phi(w)\psi^1_{r,z_0}(w)\sum_{j=2}^3
\overline{\psi^j_{r,z_0 } (w)} dV(w) \to 0\]
uniformly on $K$  as $r\to 0^+$ because
$\psi^1_{r,z_0},\psi^2_{r,z_0},\psi^3_{r,z_0}$ stay
bounded uniformly on $A_r$ for all $z_0\in K$ and
\begin{align*}
\int_{A_r} |\phi(w)|\left|\psi^1_{r,z_0}(w)\right|  dV(w)
= &-\frac{2r^2}{(1-r^2) |z_0|^2}\int_r^1\frac{\log\rho}{\rho}  d\rho \\
=&\frac{r^2(\log r)^2}{(1-r^2)|z_0|^2} \to 0
\end{align*}
uniformly on $K$ as $r\to 0^+$. Finally, since
$\psi^3_{r,z_0}$ stays bounded uniformly on $A_r$ while
$\sup\{|\psi^2_{r,z_0}(w)|:z_0\in K,w\in A_r\}\to 0$
 as $r\to 0^+$ we get
\[\int_{A_r}\phi(w) \left|\psi^2_{r,z_0}(w)\right|^2 dV(w)\to 0\] 
and 
\[ \int_{A_r}\phi(w)\psi^2_{r,z_0}(w) \overline{\psi^3_{r,z_0 } (w)}  dV(w) \to 0\]
uniformly on $K$ as $r\to 0^+$. Therefore, we showed that
\[\sup\left\{\left|\int_{A_r}\phi(w)\Psi_{r,z_0}(w)dV(w)\right|:
    z_0\in K\right\} \to 0 \text{ as } r\to 0^+.\]

Now we turn to $\int_{A_r}\phi(w)\left|\psi^0_{r,z_0}(w)\right|^2 dV(w)$.
\begin{align*}
\int_{A_r}\phi(w)\left|\psi^0_{r,z_0}(w)\right|^2 dV(w)
=\frac{1}{2\pi |z_0|^2(\log r)^2}\int_r^1 \frac{\log \rho}{\rho} d\rho
=-\frac{1}{4\pi |z_0|^2}.
\end{align*}
Finally,
\[K^{A_r}(z_0,z_0)\to K^{\mathbb{D}}(z_0,z_0)
    =\frac{1}{\pi(1-|z_0|^2)^2}\]
uniformly for all  $z_0\in K$ as $r\to 0^+$ and
\[\sup\left\{\left|\psi^3_{r,z_0}(w)\right|^2
- \left|K^{\mathbb{D}}(w,z_0)\right|^2:z_0\in K,
    w\in \mathbb{D}\right\}\to 0\]
as $r\to 0^+$. Therefore, we have
 \begin{align*}
B_{A_r}\phi(z_0)
=& \int_{A_r}\phi(w)\frac{|K^{A_r}(w,z_0)|^2}{K^{A_r}(z_0,z_0)}dV(w)\\
=& \int_{A_r}\phi(w)\frac{|\psi^0_{r,z_0}(w)|^2}{K^{A_r}(z_0,z_0)}dV(w)
+\int_{A_r}\phi(w)\frac{|\psi^3_{r,z_0}(w)|^2}{K^{A_r}(z_0,z_0)}dV(w)\\
&+\int_{A_r}\phi(w)\frac{\Psi_{r,z_0}(w)}{ K^{A_r}(z_0,z_0)}dV(w)
 \end{align*}
and
\[B_{A_r}\phi(z_0)\to -\frac{(|z_0|^2-1)^2}{4|z_0|^2}+B_{\mathbb{D}}\phi(z_0)
=\frac{|z_0|^2}{4}-\frac{1}{4|z_0|^2}\]
uniformly on $K$ as $r\to 0^+$  because Lemma \ref{Lemma2} implies that
$B_{\mathbb{D}}\phi(z_0)=\frac{1}{2}(|z_0|^2-1)$. Therefore, we showed that
\[B_{A_r}\phi(z)\to \frac{|z|^2}{4}-\frac{1}{4|z|^2}\]
uniformly on compact subsets of $\mathbb{D}\setminus\{0\}$ as $r\to 0^+$.
\end{proof}

\begin{proof}[Proof of Proposition \ref{PropLp}]
The functions $\{e_n:n=0,1,2,\ldots\}$ form an orthonormal basis
for $A^2(\disk^*)$ where $e_n(z)=\sqrt{\frac{n+1}{\pi}}z^n$.
Using integration by parts, we compute
\[T_{\phi}e_n(z)
=\left(2(n+1)\int_0^1 r^{2n+1}\log rdr\right)z^n
=-\frac{z^n}{2n+2}
=-\frac{\sqrt{\pi}}{2(n+1)^{3/2}} e_n(z).\]
Hence, $T_{\phi}$ is a compact diagonal operator on $A^2(\disk^*)$
and by \cite[Theorem 6]{Englis92} it is in the Toeplitz algebra.

Let $f(z)=\frac{|z|^2}{4}-\frac{1}{4|z|^2}$. Proposition \ref{PropAsymp}
implies that for any $\ep>0$ and any compact set
$K\Subset \mathbb{D}\setminus \{0\}$ we can choose  $r_0>0$
sufficiently small so that $K\Subset A_r$ and
\begin{align*}
\|EB_{A_r}T^{A_r}_{\phi}\|^p_{L^p(\disk^*)}
&=\int_{A_r}|B_{A_r}\phi(z)|^pdV(z)
\geq \int_K|B_{A_r}\phi(z)|^pdV(z)\\
&\geq \int_K|f(z)|^pdV(z)-\ep
\end{align*}
for all $0<r\leq r_0$. Then
\[\liminf_{r\to 0^+}\|EB_{A_r}T^{A_r}_{\phi}\|^p_{L^p(\disk^*)}
\geq \|f\|_{L^p(K)}^p-\ep.\]
Since $K$ and $\ep$ are arbitrary, we conclude that
\[\liminf_{r\to 0^+}\|EB_{A_r}T^{A_r}_{\phi}\|^p_{L^p(\disk^*)}\\
\geq \|f\|^p_{L^p(\disk^*)}.\]
Furthermore, one can show that
$\|f\|_{L^p(\disk^*)}=\infty$ if and only if $p\geq 1$.
Therefore,
\[\lim_{r\to 0^+}\|EB_{A_r}T^{A_r}_{\phi}\|_{L^p(\disk^*)}=\infty.\]
Finally, $\|B_{\disk^*}T_{\phi}\|_{L^p(\disk^*)}<\infty$
for all $1\leq p\leq \infty$ because Lemma \ref{Lemma2} implies that
 $B_{\disk^*}T_{\phi}= (|z|^2-1)/2$.
\end{proof}


\end{document}